\documentclass[11pt]{article}

\usepackage{fullpage}
\usepackage{diagbox}
\usepackage{mathtools}
\usepackage{bbm}
\usepackage{latexsym}
\usepackage{epsfig}
\usepackage{amsmath,amsthm,amssymb,enumerate}

\usepackage[a-1b]{pdfx}

\usepackage{color}

\def\teps{\sim_\e}

\def\nn{\nonumber}
\def\a{\alpha} \def\b{\beta} \def\d{\delta} \def\D{\Delta}
\def\e{\varepsilon} \def\f{\phi}   \def\g{\gamma}
  
\def\z{\zeta} \def\th{\theta}    \def\l{\lambda}
 \def\m{\mu}  \def\p{\pi}
\def\r{\rho}  \def\s{\sigma} 
\def\t{\tau} \def\om{\omega}

\newtheorem{theorem}{Theorem}
\newtheorem{lemma}[theorem]{Lemma}

\newtheorem{Remark}{Remark}

\newcommand{\wh}[1]{\widehat{#1}}

\newcommand{\rdup}[1]{{\left\lceil #1\right\rceil }}

\newcommand{\brac}[1]{\left(#1\right)}

\newcommand{\bfrac}[2]{\left(\frac{#1}{#2}\right)}

\def\cE{{\cal E}}

\newcommand{\set}[1]{\left\{#1\right\}}
\def\sm{\setminus}

\def\es{\emptyset}

\def\E{\mathbb{E}}

\def\Pr{\mathbb{P}}

\newcommand{\ignore}[1]{}

\def\cA{{\mathcal A}}
\def\cB{{\mathcal B}}

\def\cE{{\mathcal E}}

\def\cW{{\mathcal W}}

\newcommand{\beq}[2]{\begin{equation}\label{#1}#2\end{equation}}
\newcommand{\mults}[1]{\begin{multline*}#1\end{multline*}}

\def\nn{\nonumber}

\usepackage{tikz}
\usetikzlibrary{decorations.pathmorphing}
\usetikzlibrary{positioning}
\usetikzlibrary{arrows,automata}
\usetikzlibrary{shapes.misc}
\usetikzlibrary{backgrounds}
\usetikzlibrary{arrows,shapes}

\newcommand{\BD}{\mathrm{BD}}
\newcommand{\DB}{\mathrm{DB}}
\begin{document}

\date{}
\title{The Moran process on a random graph}
\author{Alan Frieze\thanks{Research supported in part by NSF grant DMS1952285 Email: frieze@cmu.edu} and Wesley Pegden\thanks{Research supported in part by NSF grant DMS1700365.Email: wes@math.cmu.edu}\\Department of Mathematical Sciences\\Carnegie Mellon University\\Pittsburgh PA 15213. }
\maketitle
\begin{abstract}
We study the fixation probability for two versions of the Moran process on the random graph $G_{n,p}$ at the threshold for connectivity. The Moran process models the spread of a mutant population in a network. Throughtout the process there are vertices of two types, mutants and non-mutants. Mutants have fitness $s$ and non-mutants have fitness 1. The process starts with a unique individual mutant located at the vertex $v_0$. In the Birth-Death version of the process a random vertex is chosen proportionally to its fitness and then changes the type of a random neighbor to its own. The process continues until the set of mutants $X$ is empty or $[n]$. In the Death-Birth version a uniform random vertex is chosen and then takes the type of a random neighbor, chosen according to fitness. The process again continues until the set of mutants $X$ is empty or $[n]$. The {\em fixation probability} is the probability that the process ends with $X=\emptyset$.

We show that asymptotically correct estimates of the fixation probability depend only on the degree of $v_0$ and its neighbors. In some cases we can provide values for these estimates and in other places we can only provide non-linear recurrences that could be used to compute values.
\end{abstract}
\section{Introduction}
Consider a fixed population of $N$ individuals of types $A$ and $B$, where the relative fitness of individuals of type $B$ is given by a real number $s$.  The classical \emph{Moran process} \cite{Moran} models  a discrete-time process for the fixed-size population where at each step, one individual is chosen for reproduction with probability proportional to its fitness ($s$ for Type B, 1 for Type A), and then replaces an individual chosen uniformly for death. In particular, for the numbers $N_A$ and $N_B$ of individuals of each type, at each step of the process, the probability that $N_B$ increases by 1 and $N_A$ decreases by 1 is
\[
p^+=\bfrac{r N_B}{N_A+r N_B} \bfrac{N_A}{N_A+N_B},
\]
and the probability that $N_B$ decreases by 1 and $N_A$ increases by 1 is
\[
p^-=\bfrac{N_A}{N_A+r N_B} \bfrac{N_B}{N_A+N_B}.
\]
(With probability $1-p^+-p^-$, $N_A$ and $N_B$ remain unchanged.)

This process was generalized to graphs by Liberman, Hauert and Nowak in \cite{iso}, see also \cite{book}.  In this setting, a fixed graph---whose vertices represent the fixed population---has vertex colors that evolve over time representing the two types of individuals.  We will focus on the case where one vertex begins colored as Type B---the mutant type.  In the \emph{Birth-Death} process, a random vertex $v$ is chosen for reproduction, with vertices chosen with probabilities proportional to the fitness of their color type, and then a random neighbor $u$ of $v$ is chosen for death, with the result that $u$ gets recolored with the color of $v$.  In the \emph{Death-Birth} process, a vertex $v$ is chosen uniformly randomly for death, and then a vertex $u$ is chosen randomly from among its neighbors, with probabilities proprotional to the fitness of their respective types, with the result that $v$ is recolored with the color of $u$.

The \emph{isothermal theorem} of \cite{iso} implies that if $G$ is a regular, connected, undirected graph, the fixation probability for Type B---that is, the probability that all vertices will eventually be of Type B---depends only on the number $n$ of vertices in $G$ and the relative fitness $s$, and not, for example on the particular regular graph or the particular choice of starting vertex (more generally, the same holds for doubly stochastic---so called \emph{isothermal}---weighted digraphs).  But it is also observed in \cite{iso} that beyond this special setting, the graph structure can have a dramatic effect on the fixation probability.  In the classical Moran model for a population of size $N$ (equivalent to the graph process when the graph is a complete graph on $N$ vertices, with loops), the fixation probability for Type B is
\[
\frac{1-s^{-1}}{1-s^{-N}}\sim 1-\frac 1 s
\]
(where $a_N\sim b_N$ indicates that $a_N/b_N\to 1$ as $N\to \infty$).
But the fixation probability for an $N$-vertex star graph is, asymptotically in $N$, for $s>1$,
\[
1-\frac 1 {s^2}
\]
(see also \cite{Chalub,Broom}).

For the special case when $s=1$ the fixation probability can be characterized in terms of coalescence times for random walks \cite{coalescence}, but for $s>1$ it is unknown whether a polynomial time algorithm exists to determine the fixation probability given a particular input graph \cite{complexity}; in place of this, heuristic approximations (e.g., see \cite{oana}) or numerical experiments (e.g., \cite{randomtests}) are used to estimate fixation probabilities. It is known however, that the expected {\em absorption time} i.e. the time when either all vertices are mutant, or all vertices are non-mutant, is polynomially bounded. In fact, Goldberg, Lapinskas and Richerby \cite{GLR} prove that the expected absorption time on an $n$-vertex graph is $O(n^{3+o(1)})$, improving results of Diaz, Goldberg, Mertzios, Richerby, Serna and Spirakis \cite{DGMRSS}. Given this, one can estimate the fixation probability for a graph by simply running the process for a sufficient number of times. The paper \cite{GLR} describe a randomized algorithm that estimates the fixation probability with a multiplicative factor of $1\pm\e$ and runs in time $O(n\bar d+\D^2\bar d\e^{-2}\log(\bar d\e^{-1}))$, where $\D$ denotes maximum degree and $\bar d$ denotes average degree.

In place of analyzing fixed graphs, one can analyze the fixation probability for a random graph from some distribution.  For the Erd\H{o}s-R\'enyi random graph $G_{n,p}$, each possible edge among a set of $n$ vertices is included independently, each with probability $p$.  For the case where $0<p<1$ is a constant independent of $n$, Adlam and Nowak leveraged the near-regularity of such graphs (in particular, that they are ``nearly-isothermal'') to show that the fixation probability on such graphs is approximated by that of the classical Moran model. When $p$ is not a constant but ``small'', in the sense that $p=p(n)\to 0$ as $n\to \infty$, $G_{n,p}$ can exhibit significant diversity in vertex degrees, and numerical experiments conducted by Mohamadichamgavi and Miekisz \cite{physics} showed a strong dependence of the fixation probability of the degree of the initial mutant vertex.

In this paper we give a rigorous analysis of fixation probabilities for random graphs with degree heterogeneity.   In particular, our first result concerns $G_{n,p}$ when $p=p(n)=\frac{\log n+\omega(1)}{n}$.  When $\omega(1)$ refers to a slow-growing function, this places our analysis right at the threshold for connectivity of $G_{n,p}$; indeed, in this regime, $G_{n,p}$ is connected w.h.p. ({\em with high probability} i.e. with probability tending to 1 as $n\to \infty$),   and most vertices have degree close to $\log n$, but still there are $\sim \log n$ vertices whose degree is just 1, see for example Frieze and Karo\'nski \cite{FK}.  Also, w.h.p., the automorphism group of the graph is trivial---that is, any two vertices can be distinguished by their relations in the network structure (even if they have the same degree), see for example Erd\H{o}s and R\'enyi \cite{ER1}.  Nevertheless, we prove that the degree of the initial mutant vertex (or possibly one of its neighbors) is enough to asymptotically determine the fixation probability on $G_{n,p}$: the following are defined w.r.t. $G=G_{n,p}$. $d(v)$ denotes the degree of vertex $v$.
\begin{align}
\e&=\frac{1}{\log\log\log n}\nn\\
S_0&=\set{v:d(v)\leq \frac{np}{10}}.\label{S0}\\
S_1&=\set{v:d(v)\notin I_\e=[(1\pm\e)np]}.\label{S1}
\end{align}
\begin{theorem}\label{th1}
Given a graph $G$ and a vertex $v_0$, we let $\f=\phi^{\BD}_{G,v_0,r}$ denote the fixation probability of the Birth-Death process on $G$ when the process is initialized with a mutant at $v_0$ of relative fitness $s>1$.  If $G$ is a random graph sampled according to the distribution $G_{n,p}$, then w.h.p., $G$ has the property that 
\begin{enumerate}[(a)]
\item If $d(v_0)=o(np)$ then $\f=1-o(1)$.
\item Suppose that $v_0\notin S_1$ and $N(v_0)\cap S_0=\es$. Then w.h.p. $\f \sim(s-1)/s$. (This includes the case where $v_0$ is chosen uniformly from $[n]$.
\item Suppose that $v_0\in S_1$. Then $\f$ depends asymptotically only on $d(v_0),s$. 
\item Suppose that $v_0\notin S_1$ and $N(v_0)\cap S_0=\set{y_0}$ holds. Then $\f$ depends asymptotically only on $d(y_0),s$. 
\end{enumerate}
On the other hand, if $s\leq 1$ then $G$ has the property that $\f_{G,v_0}=o(1)$ regardless of $v_0$. 
\end{theorem}
Here, $o(f)$ denotes a function of $n$ whose ratio to the function $f=f(n)$ has a limit of 0 as $n\to \infty$, and so, for example, part (a) asserts that for any functions $f(n),g(n)$ such that that $\lim_{n\to \infty} f(n)/np=0$ and $\lim_{n\to \infty}g(n)=0,$ the probability that $G_{n,p}$ has the property that all of its vertices $v_0$ of degree $\leq f(n)$ have fixation probability $\phi_{v_0}>1-g(n)$ has a limit of 1 as $n\to \infty$.
\begin{theorem}\label{th2}
    Given a graph $G$ and a vertex $v_0$, we let $\phi^{\DB}_{G,v_0,s}$ denote the fixation probability of the Death-Birth process on $G$ when the process is initialized with a mutant at $v_0$ of relative fitness $s>1$.  If $G$ be a random graph sampled according to the distribution $G_{n,p}$, then w.h.p., $G$ has the property that 
\begin{enumerate}[(a)]
\item Suppose that $v_0\notin S_1$ and $N(v_0)\cap S_0=\es$. Then w.h.p. $\f \sim(s-1)/s$. 
\item Suppose that $v_0\in S_1$. Then $\f$ depends asymptotically only on $d(v_0),s$. 
\item Suppose that $v_0\notin S_1$ and $N(v_0)\cap S_0=\set{v_1}$ holds. Then $\f$ depends asymptotically only on $d(v_1),s$. 
\end{enumerate}
On the other hand, if $s\leq 1$ then $G$ has the property that $\f^{\DB}_{G,v_0,s}=o(1)$ regardless of $v_0$.
\end{theorem}
\begin{Remark}
In our proofs we establish recurrence relations that enable us to asymptotically determine the fixation probabiliy in the above cases where it is not explicitly given. Unfortunately, these recurrences are non-linear, although in principle we could obtain numerical results from them.
\end{Remark}

\section{Notation} For a set $S\subseteq[n]$, we let $\bar S=[n]\setminus S$, $e(S)=|\set{vw\in E(G):v,w\in S}|$. If $S,T\subseteq [n],S\cap T=\emptyset$, then $e(S:T)=|\set{vw:v\in S,w\in T}|$. $N(S)=\set{w\notin S:\exists v\in S\ s.t.\ vw\in E(G)}$, where we shorten $N(\set{v})$ to $N(v)$. We let $N_T(S)=N(S)\cap T$ for $T\cap S=\es$. We let $d_S(v)=|N(v)\cap S|$ and $d(v)=d_{[n]}(v)$ and $\D=\max\set{d(v):v\in [n]}$. 

We let 
\[
n_1=n-\frac{n}{(np)^{1/2}}\text{ and }\om_0=\frac{\e^2np}{100\log np}.
\]
We write $A\teps B$ if $A\in (1\pm O(\e))B$ as $n\to\infty$ and $A\lesssim_\e B$ if $A\leq (1-O(\e))B$.
\paragraph{Chernoff Bounds} We use the following inequalities for the Binomial random variable $B(N,p)$:
\begin{align}
\Pr(B(N,p)\leq (1-\th)Np)&\leq e^{-\th^2np/2}\qquad 0\leq \th\leq 1.\label{ch1}\\
\Pr(B(N,p)\geq (1+\th)Np)&\leq e^{-\th^2np/3}\qquad 0\leq \th\leq 1.\label{ch2}\\
\Pr(B(N,p)\geq \l Np)&\leq \bfrac{e}{\l}^{\l Np}\qquad \l>0.\label{ch3}
\end{align}
For a proof of these bounds, see for example \cite{FK}, Part V.
\section{Random Graph Properties}
Assume that $np=O(\log n)$. We will deal with the simpler case of $np/\log n\to \infty$ in Section \ref{np>>logn}. 
\begin{lemma}\label{lem1}
The following hold w.h.p.
\begin{enumerate}[(a)]
\item $\D\leq 5np$.\label{item_max-degree}
\item $|S_1|\leq n^{1-\e^2/4}$.\label{itemsizeS1}
\item For all cycles $C$ of length at most $\om_0$ we have that $C\cap S_1=\emptyset$.\label{itemCcapS1}
\item For all $v,w\in S_0,\,v\neq w$ we have $dist(v,w)\geq \om_0$.\label{itemS0S0dist}
\item For all $v\in S_1$ and all vertices $w\neq v$ such that $d(w)<\om_0$ we have that $dist(v,w)>\om_0$. \label{itemS0S1dist}
\item For all $S\subseteq [n]$ with $|S|\leq 2n/(np)^{9/8}$ we have that $e(S)<10|S|$.\label{e(S)3|S|}
\item For all $S\subseteq [n]$ with $|S|\leq2\om_0$ we have that $e(S)\leq |S|$.\label{iteme(S)1|S|}
\item $\not\exists S\subseteq \bar S_1$ such that \label{iteme(S,T)} 
\begin{enumerate}[(i)]
\item $|S|\in I_{(d)}=[10/\e^3, n_1]$.
\item $e(S:T)\notin (1\pm2\e)|S|\,|T|p$ where $T=\bar S\sm S_1$.
\end{enumerate}
\item For all $S\subseteq [n]$ such that $S$ induces a connected subgraph and such that $\om_0/2\leq |S|\leq n/(np)^{9/8}$ we have that $e(S:\bar S)\geq|S|np/2$.\label{iteme(SSbar)}
\item For all $S\subseteq[n]$ such that $S$ induces a connected subgraph we have that $S\cup N(S)$ contains at most  $\frac{7}{\e^2}\max\set{1,\frac{|S|}{\om_0}}$ members of $S_1\cup N(S_1)$.\label{item7/eps^2}
\item For $S\subseteq [n]$, let $B_k(S)$ be the set of vertices $v\in \bar S$ with $d_S(v)=k$. Then for all $S\subseteq [n]$ such that $|S|\leq n/(np)^{9/8}$ and such that $S$ induces a connected subgraph, we have that $|B_k(S)|\leq \a_k|S|np$ for $2\leq k\leq (np)^{1/3}$,  where $\a_k=\e/k^2$.\label{itemBk}
\item If $n/(np)^2\leq |S|\leq n_1$, then there are at most $\th|S|$ vertices $v\in S$ that $d_{\bar S}(v)\notin (1\pm\e)(n-|S|)p$, where $\th=\frac{1}{e^2(np)^{1/2}}$.\label{itemtheta}
\item There do not exist disjoint sets $S,T\subseteq [n]$ with $n/(np)^{9/8}\leq |S|\leq n/(np)^{1/3}$ and $|T|=\th(n-|S|)$ such that $e(S,T)\geq \a |S||T|p,$ where $\a=(np)^{1/4}$.\label{itemSTD}
\item There do not exist $S\subseteq[n],v\in \bar S$ such that $|S|\leq np$, $S$ induces a connected subgraph and $d_S(v)\geq \e^{-2}\log\log n$.\label{itemSvd}
\end{enumerate}
\end{lemma}
\begin{proof}
We defer the proof of this lemma to Section \ref{A} in an appendix.
\end{proof}

\section{Birth-Death}\label{BirDe}
$X$ denotes the set of mutant vertices and $w(X):=(s-1)|X|+n$. We have 
\begin{align}
p_+=p_+^{BD}(X)=\Pr(|X|\to |X|+1)&=\frac{s}{w(X)}\sum_{v\in X}\frac{d_{\bar X}(v)}{d(v)}.\label{1}\\
p_-^{BD}(X)=\Pr(|X|\to |X|-1)&=\frac{1}{w(X)}\sum_{v\in N(X)}\frac{d_{X}(v)}{d(v)}.\label{2}
\end{align}

\subsection{The size of $(X\cup N(X))\cap S_1$} \label{eii}
An iteration will be a step of the process in which $X$ changes. We prove some lemmas that will be useful in later sections. In the following $Z$ is a model for  $|X|$. 
\begin{lemma}\label{link}
Suppose that $Z=Z_t,t\geq t_0$ is a random walk on $0,1,\ldots,n$ and that we have $t\geq t_0$ implies that $\Pr(Z_{t+1}=Z_t+1)\geq \g$ where $\g>1/2$, as long as $Z_t\geq \r t$, where $\r>0$ is sufficiently small. Suppose that $n$ is large and that $0,n$ are absorbing states. Suppose that for values $a,b>2a$ we have $Z_0=2a\geq 2\r t_0$. Then with probability $1-O(e^{-\Omega(a)})$, $Z$ reaches $b$ before it reaches $a$.
\end{lemma}
\begin{proof}
Let $\s=\g/2+1/4$ and let $\cE_t$ be the event that $Z$ makes at least $(t-t_0)\s-a/2$ positive moves at times $t_0+1,\ldots,t$. If $\cE_\t$ occurs for $t_0\leq \t\leq t$ then $Z_t\geq a+2(t-t_0)\s-(t-t_0)=a+(\g-1/2)(t-t_0)>\max\set{a, \r  t}$ for $\r$ sufficiently small. (This is true by assumption for $t=t_0$ and the LHS increases by $\g-1/2>\r$ as $t$ increases by one.) Let $t_1=t_0+(b-a)/(\g-1/2)$. If $\cE=\bigcap_{\t=t_0}^{t_1}\cE_\t$ occurs then $Z_{t_1}\geq b$. The Chernoff bounds imply that 
\[
\Pr(\neg\cE)\leq \sum_{\t=t_0+a/2\s}^{t_1}\Pr(\neg\cE_\t)\leq \sum_{\t=t_0+a/2\s}^{t_1}\exp\set{-\frac{1}{2}\brac{\frac{\g}{2}-\frac{1}{4}}^2\g(\t-t_0)}\leq e^{-\Omega(a)}.
\]
\end{proof}

\begin{lemma}\label{early}
While, $|X|\leq n/(np)^{9/8}$, the probability that $(X\cup N(X))\cap S_1$ increases in an iteration is $O(\e^{-2}/\om_0)$.
\end{lemma}
\begin{proof}
The probability estimates are conditional on there being a change in $X$. 

Let $S_1^+=S_1\cup N(S_1)$. We consider the addition of a member of $S_1$ to $X\cup N(X)$. This would mean the choice of $v\in X$ and then the choice of a neighbor of $v$ in $S_1^+$. Since $v$ is at distance at most 2 from $S_1$, Lemma \ref{lem1}(\ref{itemS0S1dist}) implies that $d(v)\geq \om_0$. Let $C$ be the component of the graph $G_X$ induced by $X$ that contains $v$. Assume first that $|C|\leq \om_0/2$. Then $v$ has at least $\om_0/2$ neighbors in $\bar X$ and so we see from Lemma \ref{lem1}\eqref{item7/eps^2} that the conditional probability of adding to $S_1\cap X$ is $O(\e^{-2}/\om_0)$. 

Now assume that $\om_0/2<|C|\leq n/(np)^{9/8}$. Let $C_0$ denote $\set{v\in C:d_{\bar X}(v)>0}$. We estimate the probability of adding a vertex in $S_1$ to $X\cup N(X)$, conditional on choosing $v\in C_0$. Now Lemma \ref{lem1}\eqref{e(S)3|S|} implies that $e(C)\leq 10|C|$ and Lemma \ref{lem1}\eqref{iteme(SSbar)} implies that $e(C:\bar C)\geq|C|np/2$. So, very crudely, $|C_0|\geq |C|/10$ by Lemma \ref{lem1} \ref{item_max-degree}. We write
\begin{multline}
  \label{CS0}
  \Pr(|(X\cup N(X))\cap S_1|\text{ increases }\mid \text{chosen vertex is in $C_0$, chosen neighbor is in $\bar X$})\\
 =\frac{1}{|C_0|}\sum_{v\in C_0}\frac{d_{\bar X\cap S_1^+}(v)}{d_{\bar X}(v)}=
 \frac{1}{|C_0|}\brac{\sum_{v\in C_0\cap S_0}\frac{d_{\bar X\cap S_1^+}(v)}{d_{\bar X}(v)}+\sum_{v\in C_1}\frac{d_{\bar X\cap S_1^+}(v)}{d_{\bar X}(v)}+\sum_{v\in C_2}\frac{d_{\bar X\cap S_1^+}(v)}{d_{\bar X}(v)}}
\end{multline}
where
\[
C_1=\set{v\in C_0\sm S_0:d_X(v)\geq d(v)/2}\text{ and }C_2=C_0\sm(S_0\cup C_1).
\]

The first sum in \eqref{CS0} is at most $|C_0\cap S_0|$ and it follows from Lemma \ref{lem1}\eqref{itemS0S0dist} that $|C_0\cap S_0|\leq |C|/(\om_0-1)$. It follows from Lemma \ref{lem1}\eqref{e(S)3|S|} and $d_X(v_0)\geq \log n/40$ that the second sum in \eqref{CS0} is at most $800|C|/np$. As for $C_2$, let $A_1,A_2,\ldots,A_\ell$ be the components of the graph induced by $C_2$. It follows from Lemma \ref{lem1}\eqref{item7/eps^2} that
\mults{
\sum_{v\in C_2}\frac{d_{\bar X\cap S_1^+}(v)}{d_{\bar X}(v)}=\sum_{i=1}^\ell\sum_{v\in A_i}\frac{d_{\bar A_i\cap S_1^+}(v)}{d_{\bar X}(v)}\leq \frac{20}{np}\sum_{i=1}^\ell\sum_{v\in A_i}d_{S_1^+}(v)\\
\leq \frac{20}{np}\sum_{i=1}^\ell\brac{\frac{7}{\e^2}\max\set{1,\frac{|A_i|}{\om_0}}}=O\bfrac{|C_2|}{\e^2\om_0np},
}
and we are done by \eqref{CS0} since $|C_2|\leq |C_0|$.
\end{proof}
We next consider $N(X)\cap S_0$. At any point in the process, we let $\wh X$ denote the set of vertices that have ever been in $X$ up to this point. Note that $\wh X$ induces a connected set.
\begin{lemma}\label{earlyN}
W.h.p., there are no vertices in $S_0$ added to $N(X)$ and no vertices in $N(S_0)$ added to $X$ in the first $\om_0^{3/4}$ iterations. 
\end{lemma}
\begin{proof}
Suppose that a member of $S_0$ is added to $N(X)$ because we choose $v\in X$ and then add $u\in N(v)$ to $X$, and $N(u)\cap S_0\neq \es$. It follows from Lemma \ref{lem1}(\ref{itemS0S0dist},\ref{itemS0S1dist}) that $d(v)\geq \om_0$ and the choice of $u$ is unique. So the conditional probability of this happening is $O(\om_0^{3/4}/\om_0)=o(1)$, that is conditional on there being a change in $X$. 

Suppose that we choose a $v\in X$ and add a neighbor $w\in N(S_0)$ of $v$ to $X$.  We estimate the conditional probability of this. Lemma \ref{lem1}\eqref{itemS0S0dist} rules out $v\in X$. Now $v$ cannot have two distinct such neighbors $w_1,w_2$.  Otherwise we violate Lemma \ref{lem1}\eqref{itemCcapS1} or \eqref{itemS0S0dist}. Lemma \ref{lem1}\eqref{itemS0S1dist} implies that $v$ has degree at least $\om_0$ and so the probability of choosing the unique $w$ is $O(\om_0^{3/4}/\om_0)=o(1)$.
\end{proof}
\begin{Remark}\label{rem1}
We will see in the next section that w.h.p. the size of $|X|$ follows a random walk with a positive drift in the increasing direction. It follows from this that to deal with cases where $0<|X|\leq \om$ for some $\om\leq\om_0^{1/2}$, we can assume that there will be have been at most $O(\om\log\om)$ iterations to this point. More precisely we can use the Chernoff bounds as we did in Lemma \ref{link} to argue that if $|X|$ is not absorbed at 0 then $|X|$ will reach $\om$ in $O(\om\log\om)$ iterations. Thus, given $|X|=\om$ at some point in the process we have $|\wh X|=O(\om\log\om)$.
\end{Remark}
\subsection{$p_+$ versus $p_-$}\label{pp}
In this section we bound the ratio $p_+/p_-$ for various values of $|X|$. We will see later when we analyse the cases in Section \ref{fix} that if $|X|\leq 20/\e^3$ then we only need to consider cases where $|X\cap S_1|,|N(X)\cap S_0|\leq 1$. This will in turn follow from the results of Section \ref{eii}. 

{\bf Case BD1:\; $|X|\leq 20/\e^3$ and $|X\cap S_1|,|N(X)\cap S_0|\leq 1$.}\\
Let $X_1=X\cap S_1$ and let $Y_1=N(X)\cap S_1$, $Y_0=N(X)\cap S_0$. Either $X_1=\set{x_1}$ or $X_1=\es$ and either $Y_0=\es$ or $Y_0=\set{y_0}$.  Let $X^{(i)}$ denote a connected component of $X$.

Because $|X^{(i)}|$ is small and $X^{(i)}$ induces a connected subgraph, Lemma \ref{lem1}(\ref{iteme(S)1|S|}) implies that $X^{(i)}$ induces a tree or a unicyclic subgraph. Let $\d^{(i)}_T$ be the indicator for $X^{(i)}$ inducing a tree and let $\d_T=\sum_i \d^{(i)}_T$.

Note that the number of edges inside $X$ is precisely $|X|-\delta_T\geq 0$, and the number of edges inside $X$ that are not incident to $X_1$ is precisely $|X|-\delta_T-d_X(X_1)\geq 0$.

Thus we have from \eqref{1} that
\begin{align*}
p_+&\leq \frac{s}{w(X)}\brac{\frac{(1+\e)np(|X|-|X_1|)-\big(2(|X|-\d_T-d_X(X_1))\big)}{(1-\e)np}+\frac{d_{\bar X}(X_1)}{d(X_1)}}.\\
p_+&\geq \frac{s}{w(X)}\brac{\frac{(1-\e)np(|X|-|X_1|)-\big(2(|X|-\d_T-d_X(X_1))\big)}{(1+\e)np}+\frac{d_{\bar X}(X_1)}{d(X_1)}}.
\end{align*}
We can simplify this as follows. If $|X|>|X_1|$, then $|X|-\d_T-d_X(X_1)=o(np)=o(np(|X|-|X_1|)$.  On other hand, if $|X|=|X_1|=1$, then $\d_T=1$ and $d_X(X_1)=0$.  Thus in fact we have 
\begin{align}
  p_+&\leq \frac{s}{w(X)}\brac{(1+3\e)(|X|-|X_1|)+\frac{d_{\bar X}(X_1)}{d(X_1)}}.\label{xx1}\\
p_+&\geq \frac{s}{w(X)}\brac{(1-3\e)(|X|-|X_1|)+\frac{d_{\bar X}(X_1)}{d(X_1)}}.\label{xx2}
\end{align}
We see from \eqref{2} that
\begin{align*}
p_-&\leq \frac{1}{w(X)}\brac{\frac{(1+\e)np(|X|-|X_1|)+d_{\bar X\setminus S_1}(X_1)}{(1-\e)np}+\frac{|Y_1\sm Y_0|}{np/10}+\frac{d_X(Y_0)}{d(Y_0)}}\nn\\
p_-&\geq \frac{1}{w(X)}\brac{\frac{(1-\e)np(|X|-|X_1|)-1-|X||N(X\sm X_1)\cap S_1|+d_{\bar X\setminus S_1}(X_1)}{(1+\e)np}+\frac{|Y_1\sm Y_0|}{5np}+\frac{d_X(Y_0)}{d(Y_0)}}\nn\\
\end{align*}
The -1 in the lower bound for $p_-$ accounts for vertices in $N(X)$ having more than one neighbor in $X$. It turns out from Lemma \ref{lem1}\eqref{iteme(S)1|S|} that there can be at most one such vertex and this will have two neighbors in $X$. Also $|X||N(X\sm X_1)\cap S_1|\leq |X||N(\wh X\sm X_1)\cap S_1|=O(|X|\e^{-3}\log1/\e)$ is a crude upper bound on the number of edges between $X\sm X_1$ and $N(X\sm X_1)\cap S_1$, see Remark \ref{rem1} and Lemma \ref{lem1}(\ref{item7/eps^2}). Because $np\gg \e^{-4}\log1/\e$ we can absorb the $|X||N(X\sm X_1)\cap S_1|$ into an error term.  (When $X=X_1$ this term goes away regardless.)  A similar application of Remark \ref{rem1} and Lemma \ref{lem1}(\ref{item7/eps^2}) also implies that $|Y_1\sm Y_0|/np$ is $o(\e|X|)$.  This can clearly be absorbed into the error term from $|X|>1$. When $|X|=1$ it's contribution after dividing by $w(X)$ will be $o(p_+)$ and it can be ignored. Thus we write
\begin{align}
p_-&\leq \frac{1}{w(X)}\brac{\frac{(1+\e)np(|X|-|X_1|)+d_{\bar X\setminus S_1}(X_1)}{(1-\e)np}+\frac{d_X(Y_0)}{d(Y_0)}}.\label{xx3}\\
p_--&\geq \frac{1}{w(X)}\brac{\frac{(1-\e)np(|X|-|X_1|)+d_{\bar X\setminus S_1}(X_1)}{(1+\e)np}+\frac{d_X(Y_0)}{d(Y_0)}}.\label{xx4}
\end{align} 
We now use \eqref{xx1} -- \eqref{xx4} to estimate $p_+/p_-$ in various cases.

{\bf Case BD1a: $X_1=Y_0=\es$.}\\
In this case equations \eqref{xx1}, \eqref{xx2} and $\e n p \gg 1$ imply that 
\[
\frac{s(1-3\e)|X|}{w(X)}\leq p_+\leq \frac{s(1+3\e)|X|}{w(X)}.
\]
Equations \eqref{xx3}, \eqref{xx4} imply that 
\[
\frac{(1-3\e)|X|}{w(X)}\leq p_-\leq \frac{(1+3\e)|X|}{w(X)}.
\]
So we have
\beq{y1}{
  \frac{p_+}{p_-}\teps s.
}
{\bf Case BD1b: $X_1=\set{x_1},\,Y_0=\es$ and ($|X|>1$ or $d(x_1)=\Omega(np)$).}\\
If $|X|>1$ then equations \eqref{xx3}, \eqref{xx4} imply that 
\beq{bb2}{
 \frac{(1-3\e)(|X|-1)+d(x_1)/np}{w(X)}\leq p_-\leq \frac{(1+3\e)(|X|-1)+d(x_1)/np}{w(X)}.
}
We then have, with the aid of \eqref{xx1} and \eqref{xx2} and the fact that $d(x_1)=\Omega(np)$ implies $d_{\bar X}(x_1)\teps d(x_1)$ that 
\begin{equation}\label{onev0ratio}
\frac{p_+}{p_-}\teps s\left(\frac{|X|}{|X|-1+d(x_1)/np}\right).
\end{equation}
If $|X|=1$ then $p_+=\frac{s}{w(X)}$ and $p_-\teps\frac{d(x_1)}{w(X)np}$ and so \eqref{onev0ratio} holds in this case too. 

{\bf Case BD1c: $X_1=\set{x_1},\,Y_0=\es$ and $|X|=1$ and $d(x_1)=o(np)$.}\\
We have $p_+=s/w(X)$ and \eqref{xx3}, \eqref{xx4} imply that $p_-=o(1/w(X))$. So, in this case, 
\beq{onp}{
\frac{p_+}{p_-}\teps\frac{np}{d(x_1)}\to\infty.
}
{\bf Case BD1d: $X_1=\es$ and $N(X)\cap S_0=\set{y_0}$.}\\
In this case equations \eqref{xx1}, \eqref{xx2} imply that
\beq{bb3}{
 \frac{s(1-3\e)|X|}{w(X)}\leq p_+\leq \frac{s(1+3\e)|X|}{w(X)}.
}
We have $d_X(y_0)=1$. To see this observe that $\wh X$ defined in Remark \ref{rem1} will be a connected set of size $o(\om_0)$ and so Lemma \ref{lem1}(\ref{itemCcapS1}) implies that $d_X(y_0)=1$

Equations \eqref{xx3}, \eqref{xx4} imply that 
\beq{bb4}{
\frac{1}{w(X)}\brac{(1-3\e)|X|+\frac{1}{d(y_0)}}\leq p_-\leq \frac{1}{w(X)}\brac{(1+3\e)|X|+\frac{1}{d(y_0)}}.
}
So, in this case, 
\beq{onpx}{
\frac{p_+}{p_-}\teps \frac{s}{1+\frac{1}{d(y_0)|X|}}.
}
{\bf Case BD1e: $X_1=\set{x_1}$ and $N(X)\cap S_0=\set{y_0}$.}\\
If $|X|=1$ then $p_+=\frac{s}{w(X)}$ and  $p_-\teps\frac{d(x_1)}{w(X)np}$. If $|X|>1$ then Lemma \ref{earlyN} implies that $x_1=v_0$ w.h.p. Lemma \ref{lem1}\eqref{itemS0S0dist} implies that $x_1\notin S_0$ and so $d(x_1)\geq np/10$. This means that  $d_{\bar X}(X_1)/d(X_1)\geq 1-200/\e^2np$ giving
\[
\frac{s}{w(X)}\brac{(1-3\e){X}}\leq p_+\leq \frac{s}{w(X)}\brac{(1+3\e){X}}. 
\]
Equation \eqref{bb2} is replaced by 
\beq{bb2x}{
 \frac{(1-3\e)(|X|-1)+d(x_1)/np+1/d(y_0)}{w(X)}\leq p_-\leq \frac{(1+3\e)(|X|-1)+d(x_1)/np+1/d(y_0)}{w(X)}.
}
But Lemma \ref{lem1}\eqref{itemS0S1dist} implies that $d(y_0)\geq \om_0$ and the term $1/d(y_0)$ is absorbed into error terms. Note that $x_1=v_0\notin S_0$ and so $d(x_1)/np\geq 1/10$. So \eqref{onev0ratio} holds.

{\bf Case BD2:\;$|X|\in I_1=[20/\e^3, n_1=n-n/(np)^{1/2}]$ and $|X|\geq \e t$ where $t$ denotes the iteration number and
\beq{dam}{
|(X\cup N(X))\cap S_1|\leq \begin{cases}1&|X|\leq \om_0^{1/2}.\\O(\e^{-2}|X|/\om_0)&\om_0^{1/2}\leq |X|\leq n/(np)^{9/8}.\end{cases} 
}}
It follows from Lemma \ref{lem1}(\ref{iteme(S,T)}),(\ref{item7/eps^2}) that
\begin{align}
e(X:\bar X)\geq \min\set{e(X\sm S_1:\bar X),e(X,\bar X\sm S_1)}&\geq e(X\sm S_1:\bar X\sm S_1)\nn\\
&\geq (1-2\e)|X\sm S_1|\,|\bar X\sm S_1|p\nn\\
&\geq (1-2\e)\brac{|X|-O\bfrac{t}{\e\om_0}}\brac{|\bar X|-|S_1|}p\label{19}\\
&\geq (1-3\e)|X|\,|\bar X|p.\label{20}\\
\noalign{\text{and similarly}}
e(X:\bar X)&\leq (1+3\e)|X|\,|\bar X|p.\nn
\end{align} 
Note that to go from \eqref{19} to \eqref{20} we use $|\bar X|\gg |S_1|$ from Lemma \ref{lem1}\eqref{itemsizeS1} and the assumption that $|X|\geq \e t$.

Now, 
\begin{align}
p_+&\geq \frac{s}{w(X)}\cdot\frac{e(X\setminus S_1:\bar X)}{(1+\e)np}\geq  \frac{s(1-5\e)|X|(n-|X|)}{nw(X)}.\label{1a}\\
p_+&\leq \frac{s}{w(X)}\brac{\frac{e(X\sm S_1:\bar X)}{(1-\e)np}+|X\cap S_1|}\leq \frac{s(1+5\e)|X|(n-|X|)}{nw(X)}.\label{1b}
\end{align}
When $|X|\leq n/(np)^{9/8}$ we use \eqref{dam}. For larger $X$ we have from Lemma \ref{lem1}\eqref{itemsizeS1} that $|X|(n-|X|)\geq \e^{-2}n|S_1|$. 

On the other hand,
\begin{align}
p_-&\geq \frac{1}{w(X)}\cdot \frac{e(X:\bar X\sm S_1)}{(1+\e)np}\geq \frac{(1-5\e)|X|(n-|X|)}{nw(X)},\label{2a}\\
p_-&\leq \frac{1}{w(X)}\brac{ \frac{e(X:\bar X\sm S_1)}{(1-\e)np}+|N(X)\cap S_1|}\leq \frac{(1+5\e)|X|(n-|X|)}{nw(X)}. \label{2b}
\end{align}
When $|X|\leq n/(np)^{9/8}$ we use \eqref{dam}.  For larger $X$ we again use $|X|(n-|X|)\geq \e^{-2}n|S_1|$. 

We see from \eqref{1a} -- \eqref{2b} that
\begin{align}
\frac{p_+}{p_-}&\geq\frac{s(1-4\e)|X|(n-|X|)}{nw(X)}\cdot \frac{nw(X)}{(1+5\e)|X|(n-|X|)}\geq (1-10\e)s.\label{3a}\\
\frac{p_+}{p_-}&\leq\frac{s(1+5\e)|X|(n-|X|)}{nw(X)}\cdot \frac{nw(X)}{(1-4\e)e(X:\bar X)}\leq (1+10\e)s.\label{3b}
\end{align}
and so
\begin{equation}\label{BD2}
  \frac{p_+}{p_-}\teps s.
\end{equation}
{\bf Case BD3: $|X|\geq n_1$}.\\
We have, very crudely, that w.h.p.,
\[
\frac{s|N(X)|}{5npw(X)}\leq p_+\leq \frac{s|N(X)|}{w(X)}\text{ and }\frac{|N(X)|}{5npw(X)}\leq p_-\leq \frac{|N(X)|}{w(X)}.
\]
So, 
\beq{100np}{
\frac{p_+}{p_-}\geq \frac{1}{5np}.
}
\subsection{Fixation probability -- Proof of Theorem \ref{th1}}\label{fix}
In this section we use the results of Section \ref{pp} to determine the asymptotic fixation problem, for various starting situations.
\subsubsection{Case analysis}
First recall the following basic result on random walk i.e. {\em Gambler's Ruin}: we consider a random walk $Z_0,Z_1,\ldots,$ on $A=\set{0,1,\ldots,m}$. Suppose that $Z_0=z_0>0$ and that if $Z_t=x>0$ then $\Pr(Z_{t+1}=x-1)=\b$ and $\Pr(Z_{t+1}=x+1)=\a=1-\b$. We assume that $0,z_1>z_0$ are absorbing states and that $\a>\b$. Let $\f$ denote the probability that the walk is ultimately absorbed at 0. Then, see Feller \cite{Fell} XIV.2,
\beq{abs}{
\f=\frac{(\b/\a)^{z_0}-(\b/\a)^{z_1}}{1-(\b/\a)^{z_1}}.
}
Feller also proves that if $D$ denotes the expected duration of the game then
\beq{dur}{
D=\frac{m}{\a-\b}\cdot\frac{1-(\b/\a)^a}{1-(\b/\a)^m}+\frac{a}{\a-\b}=O(m).
}

We next argue the following:
\begin{lemma}\label{n0}
W.h.p. either $X$ becomes empty or $|X|$ reaches size $\om=20/\e^3$ within $O(\e^{-3})$ iterations. 
\end{lemma}
\begin{proof}
Suppose that we consider a process $Z_1,Z_2,\ldots, $ such that $Z_t$ is the size of $X$ after $t$ iterations, unless $X$ becomes zero. In the latter case we use 0 as a reflecting barrier for $Z_t$. Now consider the first $\t=40s/(\e^3(2\s-1))$ steps of the $Z$ process where $\s=\frac{s}{2(s+1)}+\frac14\in (\frac12,\frac{s}{s+1})$. Given the probabilty that the walk followed by $X$ increases with probability $\sim s$, the Chernoff bounds imply that w.h.p. $Z$ makes at least $2\t\s/3$ positive steps and this means that $Z$ will at some stage reach $20/\e^3$. Going back to $X$, we see that w.h.p. either $|X|$ reaches 0 or $|X|$ reaches $20/\e^3$.
\end{proof}
\begin{lemma}\label{n1}
If $|X|$ reaches $\om=20/\e^3\to\infty$ then w.h.p. $X$ reaches $[n]$.
\end{lemma}
\begin{proof}
We first show that if $|X|$ reaches $\om$ then w.h.p. $|X|$ reaches $n_1=n-n/(np)^{1/2}$. This follows from Lemma \ref{link} with $a=10/\e^3$ and $b=n_1$, applied to the walk $Z_t=|X|$ at iteration $t$.  In particular, as long as $|X|>\max\set{a,\e t}$, the hypotheses of the lemma are satisfied with a positive bias $\teps s$, by \eqref{BD2}.

Now assume that $|X|=n_1$.  The analysis of Case BD3 shows that there is a probability of at least $\eta=(1/5np)^{n_2}$ that $X$ reaches $[n]$ after a further $n_2=n-n_1=n/(np)^{1/2}$ steps. Now consider the following experiment: when $|X|=n_1$, the walk moves right with probability at least 1/2 and left with probability at most 1/2. If it moves right then there is a probability of at least $\eta$ that $|X|$ reaches $n$ before it returns to $n_1$. If it moves left then \eqref{abs} implies that there is a constant $0<\z<1$ such that the probability of $|X|$ reaching $20/\e^3$ before returning to $n_1$ is at most $(1-\z)^{n_1}$, for any constant $\z<s/(s+1)$. (Since this event can be analyzed with Case BD2 exclusively.) Let $m=\eta^{-1}\log n$. Then we have $m(1-\z)^{n_1}\to 0$ and $m\eta\to \infty$. So w.h.p. $X$ will reach $[n]$ after at most $m$ returns and never visit $20/\e^3$. Indeed, the probability it does not reach $[n]$ is at most $o(1)+(1-\eta)^{m/2}=o(1)$. (The first $o(1)$ bounds the probability that $|X|$ moves right from $n_1$ fewer than $m/2$ times.)
\end{proof}
The next lemma concerns the case where $d(v_0)=o(np)$.
\begin{lemma}\label{vearly}
Suppose that $d(v_0)=o(np)$ and let $\om=np/d(v_0)$. Then w.h.p. $v_0\in X$ for the first $\om^{1/2}$ iterations.
\end{lemma}
\begin{proof}
If $X=\set{v_0}$ then it follows from Case BD1c that $p_-/p_+=O(1/\om)$ and the probability $X$ becomes empty is $O(1/\om)$. If $v_0\in X$ and $|X|>1$ then the probability that $v_0$ is removed from $X$ in the next iteration is also $O(1/\om)$. This is because all of $v_0$'s neighbors have degree $\Omega(np)$ outside $X$ (Lemma \ref{lem1}(\ref{itemS0S0dist})) and all vertices of $X$ other than $v_0$ have many more neighbors in $\bar X$ than $X$.  (Note from Lemma \ref{early} that no vertex in $X$ other than $v_0$ can be in $S_0$).

So, the probability that $v_0$ gets removed this early is $O(\om^{1/2}/\om)=o(1)$. 
\end{proof}

{\bf Case BDF1: $v_0\notin S_1$ and $Y_0=\es$.}\\
This includes the case where $v_0$ is chosen uniformly at random. (This follows from Lemma \ref{lem1}(\ref{itemsizeS1}).) We have $d(v_0)\teps np$ and Lemma \ref{early} implies that only Case BD1a is relevant for the first $\omega_0^{3/4}$ steps, and in this case the bias $p_+/p_-$ in the change in the size of $|X|$ is asymptotically equal to $s$. Equation \eqref{abs} implies that so long as the bias is asymptotically $s$, $|X|$ will reach $m=2\omega_0^{1/2}$ before reaching 0, with probability $\teps (s-1)/s$. Equation \ref{dur} implies that w.h.p. this happens during the first $\om_0^{3/4}$ iterations. Lemma \ref{n1} then implies that w.h.p. $X$ will reach $[n]$ from here, proving Theorem \ref{th1}(b).

{\bf Case BDF2: $d(v_0)=o(np)$:} We are initially in Case BD1c and Lemma \ref{vearly} implies that we stay in this case for the first $\om^{1/2}$ iterations, where $\om=np/d(v_0)$. Whenever $|X|=1$, we see from \eqref{onp} that the probability $X$ becomes empty in the next iteration is $O(1/\om)$. Furthermore, \eqref{abs} with $a=2$ then implies that $|X|$ reaches $\om^{1/3}$ with probability $\teps (s-1)/s$ before returning to 1. Combining the above facts, we see that $|X|$ reaches $\om^{1/3}$ within $\om^{1/2}$ iterations and we can then apply Lemma \ref{n1} to complete the proof of Theorem \ref{th1}(a).

{\bf Case BDF3: $v_0\in S_1,d(v_0)=\a np, N(v_0)\cap S_0=\es$ where $\a\neq 1$ is a positive constant.}\\
This is part of Case BD1b of Section \ref{pp}.  We consider the first $\om_0^{1/2}$ steps. It follows from Section \ref{eii} that $X\cap S_1\subseteq \set{v_0}$ throughout the first $\om_0^{1/2}$ iterations. We note that $d_{\bar X}(x_1)/d(x_1)=1-o(1)$. At iteration $j\leq \om_0$ we will have $p_+/p_-$ equal to either $(1-o(1))s/(1+(\a-1)/j)$ (by \eqref{onev0ratio}) or $(1-o(1))s$ depending on whether or not $v_0$ is still in $X$. And we note that if $v_0\in X$ then, conditioned on $X$ losing a vertex in the next iteration, the vertex it loses is $v_0$ with probability asymptotically equal to 
\begin{equation}\label{losingv0}
\frac{\frac{\alpha np}{n}\frac{1}{np}}{\frac{(|X|-1)np}{n}\frac{1}{np}+\frac{\alpha n p}{n}\frac 1 {np}}=\frac{\a}{|X|-1+\alpha}.
\end{equation}
Also, if $v_0$ leaves $X$ then Lemma \ref{early} implies that it only returns to $X$ with probability $O(\e^{-2}\om_0^{1/2}/\om_0)=o(1)$ in the next $\om_0^{1/2}$ iterations. 

We can thus asymptotically approximately model $|X|$ in the first $\om_0^{1/2}$ iterations as a random walk $\cW_0=(Z_0=1,Z_1,\ldots,)$ on $\set{0,1,\ldots,n}$ where at the $t$th step if $Z_{t-1}=j>0$ then either (i) $\Pr(Z_t=Z_{t-1}+1)=\a_j=s/(s+1+(\a-1)/j)$ or (ii) $\Pr(Z_t=Z_{t-1}+1)=\b=s/(s+1)$. The walk starts with probabilities as in (i) and at any stage may switch irrevocably to (ii) with probability $\teps \eta_j=\a/(j-1+\a)$, by \eqref{losingv0}. The fixation probability is then asymptotically equal to the probability this walk reaches $m=\om_0^{1/3}$ before it reaches 0. Let $q_j=\frac{s^{-j}-s^{-\om_0^{1/3}}}{1-s^{-\om_0^{1/3}}}$ denote the probability of reaching 0 before $m$ in the random walk $\cW_1$ where there is always a rightward bias of $s$.

Let $p_j=p_j(BDF3)$ denote the probability that the walk reaches 0 before $m=\omega_0^{1/3}$. (The BDF3 in brackets indicates that while $p_j$ always refers to the probability of the stated event , its value depends on the particular case.) Then $p_0=1$ and $p_{m}=0$ and 
\beq{pj}{
p_j=\a_jp_{j+1}+(1-\a_j)(1-\eta_j)p_{j-1}+(1-\a_j)\eta_jq_j
}
for $1<j<\om_0^{1/2}$, from which we can compute $p_1$, asymptotically.

If $|X|$ reaches $\om_0^{1/2}$ then Lemma \ref{n1} implies that it will reach $n$ w.h.p. This establishes part (c) of Theorem \ref{th1}.

{\bf Case BDF4: $v_0\notin S_1$ and $N(v_0)\cap S_0=\set{y_0}$:}\\
As such we begin in Case BD1d. We again consider the first $\om_0^{1/2}$ rounds and we see that w.h.p. we remain in this case, unless $v_0$ leaves $X$. 

Thus, as in BDF3, we can asymptotically approximately model $|X|$ in the first $\omega_0^{1/2}$ iterations as a suitable random walk, showing that the fixation probability is a function just of $d(y_0)$ in this case. Equation \eqref{pj} becomes
\beq{pjx}{
p_j=p_j(BDF4)=\b_jp_{j+1}+(1-\b_j)(1-\eta_j)p_{j-1}+(1-\b_j)\eta_jq_j
}
where $\b_j=s(s+1+1/jd(y_0))$, which comes from replacing \eqref{bb2} by \eqref{bb4}. This establishes part (d) of Theorem \ref{th1}.

{\bf Case BDF5: $v_0\in S_1$ and $N(v_0)\cap S_0=\set{y_0}$:}\\
This has the same characteristics as Case BDF3. They both rely on \eqref{onev0ratio}.
\subsubsection{$s<1$}
Arguing as above we see that except when $|X|\leq 20/\e^3$ that w.h.p. the size of $X$ follows a random walk where the probability of moving left from a positive position is asymptotically at least $\tfrac{|X|-1}{(s+1)|X|}>\frac12$ for $|X|>\tfrac{2}{1-s}$. We argue as in we did at the end of Case BDF2, with right moves and left moves reversed, that w.h.p. $X$ becomes empty.
\subsubsection{$s=1$}
It follows from Maciejewski \cite{Mac} that the fixation probability of vertex $v$ is precisely $\p(v)=d(v)^{-1}/\sum_{w\in [n]}d(w)^{-1}$. In a random graph with $np=O(\log n)$ this gives $\max_v\p(v)=O(\log n/n)$ and when $np\gg\log n$ this gives $\max_v\p(v)=O(1/n)$.

\subsection{$np\gg\log n$ and $s>1$}\label{np>>logn}
If $np/\log n\to \infty$ then all vertices have degree $\sim np$, see Theorem 3.4 of \cite{FK}. So $S_1=\es$ and all but (\ref{e(S)3|S|}), (\ref{iteme(S)1|S|}), (\ref{itemBk}), \eqref{itemtheta} of Lemma \ref{lem1} hold trivially. But \eqref{e(S)3|S|} is only used to bound $e(X:\bar X)$, where there is the possibility of low degree vertices. This is unnecessary when $np/\log n\to \infty$ since then w.h.p. $e(S:\bar S)\sim |S|(n-|S|)np$ for all $S$. Property \eqref{iteme(S)1|S|} is only used in \eqref{xx1}, \eqref{xx2} to bound $e(X)$. But because $|X|$ is small this will be small compared to $|X|np$ and only contributes to the error term. Properties \eqref{itemBk}, \eqref{itemtheta} are not used in analysing Birth-Death. In conclusion we see that only Case BDF1 is relevant and Theorem \ref{th1} holds in this case.
\section{Death-Birth} \label{DB}
The analysis here is similar to the Birth-Death process and so we will be less detailed in our description. We first replace \eqref{1}, \eqref{2} by 
\begin{align}
p_+=p_+^{DB}(X)=\Pr(|X|\to |X|+1)&=\frac{1}{n}\sum_{v\in N(X)}\frac{sd_X(v)}{sd_X(v)+d_{\bar X}(v)}.\label{101}\\
p_-=p_-^{DB}(X)=\Pr(|X|\to |X|-1)&=\frac{1}{n}\sum_{v\in X}\frac{d_{\bar X}(v)}{sd_X(v)+d_{\bar X}(v)}\leq \frac{|X|}{n}.\label{102}
\end{align}
We use the notation of Section \ref{BirDe}. We will once again assume first that $np=O(\log n)$ and remove the restriction later in Section \ref{npgglogn}.
\subsection{The size of $(X\cup N(X))\cap S_1$}\label{eiii}
\begin{lemma}\label{earlyxx}
While, $|X|\leq n/(np)^{9/8}$, the probability that $X\cap (S_1\sm S_0)$ increases in an iteration is $O(\e^{-2}/\om_0)$.
\end{lemma}
\begin{proof}
We consider the addition of a member of $S_1$ to $X$. This would mean the choice of $v\in N(X)\cap (S_1\sm S_0)$ and then the choice of a neighbor $w$ of $v$ in $X$. Let $C$ be the component of the graph $G_X$ induced by $X$ that contains $w$. Assume first that $|X|\leq np$. We have $d(v)\geq np/10$ and Lemma \ref{lem1}\eqref{itemSvd} implies that $d_S(v)\leq (\log\log n)^2$. So we can bound the probability of adding a member of $S_1\sm S_0$ by $O(\e^{-2}\log\log n/np)=O(\e^{-2}/\om_0)$.

Now assume that $np<|C|\leq n/(np)^{9/8}$. Then,
\[
\Pr(X\cap S_1\text{ increases})\leq \frac{A}{B},
\]
where
\[
A=\sum_{v\in N(X)\cap S_1}\frac{sd_X(v)}{sd_X(v)+d_{\bar X}(v)}\quad\text{ and }\quad B= \sum_{v\in N(X)}\frac{sd_X(v)}{sd_X(v)+d_{\bar X}(v)}.
\]
Now applying Lemma \ref{lem1}\eqref{item7/eps^2} to each component of $G_X$ shows that 
\[
|N(X)\cap S_1|\leq \frac{7s|X|}{\e^2\om_0}\text{ and so }A\leq \frac{7s|X|}{\e^2\om_0}.
\] 
On the other hand, Lemma \ref{lem1}\eqref{item_max-degree}\eqref{iteme(S,T)} imply that 
\mults{
Bs^{-1}\geq \sum_C\frac{|N(C\sm S_1)|-|C\cap S_1|}{5np}\geq \sum_C\frac{|C\sm S_1|(n-|C|-|S_1|)p-|C\cap S_1|}{5np}\geq\\
 \sum_C\frac{|C\sm S_1|(n-o(n))p-|C\cap S_1|}{5np}\geq \sum_C\frac{|C|(n-o(n))p-|C\cap S_1|(np+1)}{5np}\geq\\
 \sum_C\frac{|C|\brac{(n-o(n))p-\frac{7(np+1)}{\e^2\om_0}}}{5np}\geq \sum_C\frac{|C|}{6}=\frac{|X|}6.
}
\end{proof}

We next consider the first $\om_0^{3/4}$ iterations. 
\begin{lemma}\label{earlyx}
W.h.p. $N(X)\cap S_0$ does not increase during the first $\om_0^{3/4}$ iterations. 
\end{lemma}
\begin{proof}
Consider the addition of a member of $S_0$ to $N(X)$. Suppose that a member of $S_0$ is added to $N(X)$ because we choose $v\in N(X)$where $N(v)\cap S_0\neq \es$ and we then choose $w\in N(v)\cap X$. Lemma \ref{lem1}(\ref{itemS0S0dist}) implies that $d(v),d(w)\geq np/10$ and so we can bound this possibility in the first $\om_0^{3/4}$ iterations by $O(\om_0^{3/4}/np)=o(1)$.  
\end{proof}
\begin{lemma}\label{earlyXX}
W.h.p., if $d(v_0)\leq \e^{-2}$ then $d_{X}(X_1)\leq 1$ during the first $\om_0^{3/4}$ iterations. (We remind the reader that $X_1=X\cap S_1$.) Furthermore, if such a neighbor leaves $X$ then $d_{X}(X_1)=0$ for the remaining iterations up to $\om_0^{3/4}$.
\end{lemma}
\begin{proof}
After the first iteration either $X=\es$ or $X=\set{v_0,v_1}$ where $v_1\in N(v_0)$. As long as $|X|>1$, the chance of adding another neighbor of $v_0$ to $X$ is $O\bfrac{d(v_0)-1}{(|X|-1)np+d(v_0)-1}=O\bfrac{\e^{-2}}{np}$. So, the probability that $d_X(v_0)$ reaches 2 is $O(\e^{-2}\om_0^{3/4}/np)=o(1)$. The same calculation suffices for the second claim.
\end{proof}

\subsection{Bounds on $p_+,p_-$}
It follows from Lemma \ref{earlyx} that we only need to consider the case where (i) $|X|>\om_0^{1/2}$ or (ii) $|X|\leq\om_0^{1/2}$ and $X\cap S_1\subseteq \set{v_0}$ and $|N(X)\cap S_0|\leq 1$. 

{\bf Case DB1: $|X|\leq 20/\e^3$ and $|X_1|,|Y_0|\leq 1$.}\\
We remind the reader that $\wh X$ is connected and so w.h.p. if $v\in N(X)$ then $d_X(v)=1$, except possibly in one instance where $d_X(v)=2$. We write
\begin{align}
p_+&=\frac{s}{n}\brac{\sum_{v\in N(X)\sm Y_0}\frac{d_X(v)}{sd_X(v)+d_{\bar X}(v)}+\sum_{v\in N(X)\cap (S_1\sm S_0)}\frac{d_X(v)}{sd_X(v)+d_{\bar X}(v)}+\frac{d_X(Y_0)}{sd_X(Y_0)+d_{\bar X}(Y_0)}}\label{DB1eq1}\\
&\teps \frac{s}{n}\brac{\sum_{v\in N(X)\sm Y_0}\frac{1}{d_{\bar X}(v)}+\frac{|Y_0|}{sd_X(Y_0)+d_{\bar X}(Y_0)}}\nn\\
&=\frac{s}{n}\brac{\sum_{w\in X}\sum_{v\in N(w)\sm(X\cup Y_0)}\frac{1}{d_{\bar X}(v)}+\frac{|Y_0|}{sd_X(Y_0)+d_{\bar X}(Y_0)}}.\label{DB1eq1x}
\end{align}
Here we have used the fact that $d(v)\geq np/10$ and $np\gg|X|$ to remove $d_{X}(v)$  from the first two denominators. This is also used to remove the second summation in \eqref{DB1eq1}. So, separating $w\in X_1$ from the rest of $X$ we see that when $|X|>1$, (using Lemma \ref{lem1}\eqref{item7/eps^2}),
\begin{align}
p_+&\teps \frac{s}{n}\brac{|X|-|X_1|+\sum_{\substack{w\in X_1 \\ v\in N(w)\sm X}}\frac{1}{d_{\bar X}(v)}+\frac{|Y_0|}{sd_X(Y_0)+d_{\bar X}(Y_0)}}.\label{DB1a}\\
\noalign{When $|X|=1$ we have $p_-=1/n$ and when $|X|>1$}
p_-&\teps\frac{1}{n}\brac{|X|-|X_1|+\frac{d_{\bar X}(X_1)}{d_X(X_1)+d_{\bar X}(X_1)}}.\label{DB2a}
\end{align}
{\bf Case DB1a: $|X|=1$ and $X=\set{x}$:}
\begin{align}
p_+&\teps \frac{s}{n}\brac{\a+\frac{|Y_0|}{sd_X(Y_0)+d_{\bar X}(Y_0)}},\qquad\text{if $d(x)=\a np$ where $\a=\Omega(1)$}.\label{DB1aeq1}\\
p_+&\in \left[\frac{sd(x)}{5n^2p},\frac{10sd(x)}{n^2p}\right]\text{ if }d(x)=o(np).\label{DB1aeq2}\\
p_-&=\frac{1}{n}.\nn
\end{align}
{\bf Explanation for \eqref{DB1aeq1}, \eqref{DB1aeq2}:} 
Let $A=\sum_{v\in N(x)}1/d(v)$. This replaces the first sum in \eqref{DB1eq1x}. If $d(x)=\Omega(np)$ then Lemma \ref{lem1}\eqref{item7/eps^2} implies that $A\teps \a$. If $d(x)=o(np)$ then Lemma \ref{lem1}\eqref{itemS0S0dist} implies that $np/10\leq d(v)\leq 5np$ for $v\in N(x)$.

{\bf Case DB1b: $|X|>1$ and $X_1=Y_0=\emptyset$.}\\
It follows from \eqref{DB1a} that w.h.p.
\beq{DB1x}{
p_+\teps \frac{s|X|}{n}\text{ and }p_-\teps \frac{|X|}{n}.
}

{\bf Case DB1c: $|X|>1$ and $X_1=\set{x_1}$ and $d(x_1)=\a np\gg\e^{-2}$ and $Y_0=\es$.}\\ 
\beq{DB3}{
p_+\teps \frac{s(|X|-1+\a)}{n}\text{ and }p_-\teps \frac{|X|}{n}.
}
Here we have used Lemma \ref{lem1}\eqref{item7/eps^2} to replace the sum in \eqref{DB1a} by $\a$. (When we apply the lemma, the set $S$ will be the connected component of $X$ that contains $x_1$.)

{\bf Case DB1d: $|X|>1$ and $X_1=\set{x_1}$ and $d(x_1)=O(\e^{-2})$ and $Y_0=\es$.}\\ 
\beq{DB3a}{
p_+\teps \frac{s(|X|-1)}{n}\text{ and }p_-\teps\frac{1}{n}\brac{|X|-1+\frac{d(x_1)-\d_1}{s\d_1+d(x_1)-\d_1}}\sim\frac{|X|}{n},
}
where $\d_1=d_X(x_1)$. Note that Lemma \ref{earlyXX} implies that w.h.p. $x_1$ has at most one neighbor in $X$. We have used Lemma \ref{lem1}\eqref{itemS0S0dist} to remove the sum in \eqref{DB1a}.

{\bf Case DB1e: $|X|>1$ and $X_1=\es$ and $Y_0=\set{y_0}$.}\\ 
\beq{DB1c}{
p_+\teps \frac{s}{n}\brac{|X|+\frac{1}{d(y_0)+s-1}}\text{ and }p_-\teps \frac{|X|}{n}.
}
$d_X(y_0)=1$ follows from Lemma \ref{lem1}\eqref{itemCcapS1}, since $\wh X$ of Remark \ref{rem1} is connected.

{\bf Case DB1f: $|X|>1$ and $X_1=\set{v_0}$ and $Y_0=\set{y_0}$.}\\ 
In this case, if $d(v_0)=\a np$ then $\a=\Omega(1)$ since $v_0\notin S_0$, we have
\beq{DB1cx}{
p_+\teps \frac{s}{n}\brac{|X|-1+\a+\frac{1}{d(y_0)+s-1}}\text{ and }p_-\teps \frac{|X|}{n}.
}

{\bf Case DB2: $20/\e^3<|X|\leq n/(np)^{9/8}$.}\\
We assume first that $X$ induces a connected subgraph. Let $B(X)=\bigcup_{k\geq 2}B_k(X)$, where $B_k(X)=\set{v\notin X:d_X(v)=k}$, see Lemma \ref{lem1}(\ref{itemBk}). Then Lemma \ref{lem1}(\ref{e(S)3|S|}) implies that if $k>10$ then $|B_k|\leq a_k|X|$ where $a_k=10/(k-10)$. To see this, observe that if not then we can add $a_k|X|$ vertices to $X$ to make a set $Y$, such that $|Y|=(a_k+1)|X|\leq 2n/(np)^{9/8}$ with $e(Y)\geq ka_k|X|=10|Y|$, which contradicts Lemma \ref{lem1}(\ref{e(S)3|S|}).

Then, if $\wh N(X)=N(X)\setminus (B(X)\cup S_1)$ then from Lemma \ref{lem1}\eqref{iteme(S,T)}\eqref{item7/eps^2}\eqref{itemBk}and Remark \ref{rem1} (used to replace $|X|$ by $|\wh X|$ in one place,
\begin{align*}
|\wh N(X)|&\geq e(X\sm S_1,\bar X\sm S_1)-\sum_{k\geq 2}B_k(X)\\ 
&\geq (1-2\e)|X\sm S_1|(n-|X|-|S_1|)p-  \sum_{k=2}^{(np)^{1/3}}\frac{\e|X|np}{k^2}-\sum_{k=(np)^{1/3}}^{5np}\frac{10|X|}{k-10}\\
&\geq (1-3\e)|X|np-|\wh X\cap S_1|np-\frac{\e\p^2|X|np}{6}+(10+o(1))|X|\log(5np)\\
&\geq (1-4\e)|X|np.
\end{align*}
So,
\[
p_+\geq \frac1n\sum_{v\in\wh N(X)}\frac{s}{(1+\e)np}\gtrsim_\e \frac{s|X|}{n}\text{ and }p_-\leq \frac{|X|}{n}. 
\]
If $X$ induces components $C_1,C_2,\ldots,C_k$ then 
\mults{
p_+\geq \frac1n\sum_{v\in N(X)}\frac{s}{(1+\e)np}\geq \frac1n\sum_{i=1}^k\sum_{v\in N(C_i)}\frac{s}{(1+\e)np}\\ 
\geq \frac1n\sum_{i=1}^k\sum_{v\in \wh N(C_i)}\frac{s}{(1+\e)np}\gtrsim_\e\frac1n\sum_{i=1}^k\frac{s|C_i|}{(1+\e)np}=\frac{s|X|}{(1+\e)np}.
}
{\bf Case DB3: $n/(np)^{9/8}<|X|\leq n_1$.}\\
Let $D(X)=\set{v\in X:d_{\bar X}(v)\in (1\pm\e)(n-|X|)p}$. Then, using Lemma \ref{lem1}\eqref{item_max-degree}\eqref{itemsizeS1}\eqref{iteme(S,T)}\eqref{itemtheta},
\begin{align}
\sum_{v\in D(X)\sm S_{1}}\frac{d_{\bar X}(v)}{sd_{X}(v)+d_{\bar X}(v)}&\leq\sum_{v\in D(X)\sm S_1} \frac{d_{\bar X}(v)}{s((1-\e)np-(1+\e)(n-|X|)p)+(1-\e)(n-|X|)p}\nn\\
&\leq  \frac{e(D(X)\sm S_1,\bar X\sm S_1)+5np|S_1|}{(n+(s-1)|X|-((2s+1)n-(s+1)|X|)\e)p}\nn\\
&\leq \frac{(1+3\e)|X|(n-|X|)p}{(n+(s-1)|X|-((2s+1)n-(s+1)|X|)\e)p}.\label{bd1}\\
\sum_{X\cap S_1}\frac{d_{\bar X}(v)}{sd_{X}(v)+d_{\bar X}(v)}&\leq|X\cap S_1|\leq |S_1|\leq \frac{|X|}{np}.\label{bd2}\\
\sum_{X\sm(D(X)\cup S_1)}\frac{d_{\bar X}(v)}{sd_{X}(v)+d_{\bar X}(v)}&\leq\th|X|,\qquad\text{where }\th=\frac{1}{\e^2(np)^{1/2}}.\label{bd3}
\end{align}
The last inequality follows from Lemma \ref{lem1}(\ref{itemtheta}).

So we see that if $|X|=\xi n$ then after summing the above inequalities and simplifying, we see that 
\beq{DB4}{
p_-\lesssim_\e \frac{\xi(1-\xi)}{1+(s-1)\xi}.
}
We now look for a lower bound on $p_+$. 
\begin{align*}
p_+&\geq \frac{1}{n}\sum_{v\in N(X)\cap (D(\bar X)\sm S_1)}\frac{sd_X(v)}{(s-1)d_X(v)+(1+\e)np}\\
&\geq \frac{1}{n}\sum_{v\in N(X)\cap (D(\bar X)\sm S_1)}\frac{sd_X(v)}{(s-1)(1+\e)|X|p+(1+\e)np}\\
&\geq \frac{se(\bar X\sm S_1,X\sm S_1)-se(\bar X\sm D(\bar X),X)}{(s-1)(1+\e)|X|p+(1+\e)np}\\
&\geq \frac{s(1-2\e)|X|(n-|X|)p-se(\bar X\sm D(\bar X),X)}{(s-1)(1+\e)|X|p+(1+\e)np},\qquad\text{ from Lemma \ref{lem1}\eqref{iteme(S,T)}}.
\end{align*}
With $\a=1/(np)^{1/4}$ and $\th=1/\e^2(np)^{1/2}$,
\[
e(\bar X\sm D(\bar X),X)\leq \begin{cases}\a\th|X|(n-|X|)p&\frac{n}{(np)^{9/8}}\leq |\bar X|\leq  \frac{n}{(np)^{1/3}}, \quad\text{Lemma \ref{lem1}\eqref{itemtheta}\eqref{itemSTD} applied to }S=\bar X.\\5\th(n-|X|)np&\frac{n}{(np)^{1/3}}\leq |\bar X|\leq n_1,\quad\text{Lemma \ref{lem1}\eqref{itemtheta}\eqref{item_max-degree} applied to }S=\bar X.\end{cases}
\]
It follows from this that in both cases $e(\bar X\sm D(\bar X),X)\leq \e|X|(n-|X|)p$. So,
\begin{align}
p_+&\geq \frac{s(1-3\e)\xi(1-\xi)}{(s-1)(1+\e)\xi+1+\e}\\
&\gtrsim_\e \frac{s\xi(1-\xi)}{(s-1)\xi+1}.\label{DB4+}
\end{align}

In which case
\[
\frac{p_+}{p_-}\gtrsim_\e\frac{s\xi(1-\xi)}{(s-1)\xi+1}\cdot\frac{1+(s-1)\xi}{\xi(1-\xi)}= s.
\]

\subsection{Fixation probability}\label{FPT}

We first prove the equivalent of Lemma \ref{n1}.
\begin{lemma}\label{n1a}
If $|X|$ reaches $\om$, where $\om\to\infty$ then w.h.p. $X$ reaches $[n]$.
\end{lemma}
\begin{proof}
We first show that if $|X|$ reaches $\om$ then w.h.p. $|X|$ reaches $n_1=n-n/(np)^{1/2}$. Let $a=\om/2$ and $m=n_1-a$. There is a positive bias of $\teps s$ in Cases DB2, DB3 as long as $|X|>a$. It follows from \eqref{abs} that the probability $|X|$ ever reaches $a$ before reaching $m$ is $o(1)$.

Now consider the case of $|X|\geq n_1$. Comparing \eqref{2} and \eqref{101} we see that $p_+^{DB}(X)\geq p_-^{BD}(X)$.  Comparing \eqref{1} and \eqref{102} we see that  $p_-^{DB}(X)\leq P_+^{BD}(X)$ . By comparing this with Case BD3 of Section \ref{pp}, we see that this implies that $p_+^{DB}(X)/p_-^{DB}(X)\geq 1/(5snp)$. Now consider the experiment described in Lemma \ref{n1}. Beginning with $|X|=n_1$, we still have a probability of at most $(1-\z)^{n_1}$ of $|X|$ reaching 0 before returning to $n_1$. Now there is a probability of at least $\eta=(5snp)s^{-n/(np)^{1/2}}$ of $|X|$ reaching $n$ before returning to $n_1$. 
\end{proof}
\subsubsection{Case analysis}
We consider the following cases:\\
{\bf Case DBF1: $v_0\notin S_1$ and $Y_0=\es$.}\\
In this case Lemmas \ref{earlyxx} and \ref{earlyx} imply that we remain in Case DB1a or DB1b while $|X|\leq \om_0^{1/2}$ and there is a bias to the right $p_+/p_-\teps s$. (Lemma \ref{earlyx} also implies that $X\cap S_0$ remains empty. In this case, before adding to $X\cap S_0$ we must add to $N(X)\cap S_0$.) Remark \ref{rem1} implies that w.h.p. we either reach $|X|=0$ or $|X|=\om_0^{1/2}$ within $O(\om_0^{1/2}\log\om_0)$ iterations. If $|X|$ reaches $\om_0^{1/2}$ then Lemma \ref{n1a} implies that w.h.p. $X$ eventually reaches $[n]$.
Consequently $\f\teps (s-1)/s$. This proves part (a) of Theorem \ref{th2}.

{\bf Case DBF2: $v_0\in S_1, Y_0=\es$ and $d(v_0)=\a np$ where $\a np\gg \e^{-2}$.}\\
In this case we remain in Case DB1a or DB1c while $|X|\leq \om_0^{1/2}$ as long as $v_0$ is not removed from $X$. If $d(v_0)=\a np$ then the probability that this happens, conditional on a change in $X$, is $\teps s/((s+1)|X|-1+\a)$. There are $|X|$ chances of about $s/n$ of choosing $v\in X$. Then for each $w\in X\sm\set{v_0}$ there is a chance of about $1/n$ that $v$ is a neighbor of $w$ and that $v$ chooses $w$ as $u$. This leads to \eqref{pj} with $\eta_j=s/((s+1)j-1+\a)$ and gives $p_j(DBF2)$. 

{\bf Case DBF3: $v_0\in S_1, Y_0=\es$ and $d(v_0)=O(\e^{-2})$.}\\
In this case the term $\psi(\d_1)=\frac{d(x_1)-\d_1}{s\d_1+d(x_1)-\d_1}$ in \eqref{DB3a}, where $\d_1=d_X(x_1)$,  may become significant. It is only significant while $v_0\in X$ and $|X|\leq \om_0^{1/2}$. In which case $\d_1$ is either 0 or 1. If $\d_1=1$ and $|X|=j\geq 2$ then conditional on $|X|$ decreasing, $\d_1$ becomes 0 with asymptotic probability $1/j$. If $\d_1=0$ and $|X|=j\geq 2$ then $\d_1$ becomes 1 with asymptotic probability 0. If $|X|=1$ and $X=\set{v_0}$ then $\d_1$ becomes 1 if and only if $X$ does not become empty after the next iteration. This leads to the following recurrence: let $p_{j,\d}$ be the (asymptotic) probability of $|X|$ becoming 0 starting from $|X|=j$ and $\d_1=\d$. Then we have $p_{0,0}=p_{0,1}=1$ and $p_{m,\d}=0$ for $m=\om_0^{1/2}$ and $\d=0$ or 1. The recurrence is
\begin{align*}
p_{j,0}&=\g_jp_{j+1,0}+(1-\g_j)(1-\eta_j)p_{j-1,0}+(1-\g_j)\eta_jq_j.\\
p_{j,1}&=\g_jp_{j+1,1}+(1-\g_j)(1-\eta_j)(1-\th_j)p_{j-1,1}+(1-\g_j)(1-\eta_j)\th_jp_{j-1,0}+(1-\g_j)\eta_jq_j.
\end{align*}
Here $\g_j=(s(j-1))/(s(j-1)+j-1+\psi(\d))$ is asymptotic to the probability that $j=|X|$ increases, $\eta_j=s/(sj+j-1)$ is asymptotic to the probability that $v_0$ leaves $X$ and $\th_j=\eta_j$ is asymptotic to the probabilty that $v_0$'s neighbor in $X$ leaves $X$. (We have $\a=1$ in the definition of $\eta_j$, since $v_0\notin S_1$.) $\eta_j,q_j$ are as in \eqref{pj}. This, with the previous case,  establishes part (b) of Theorem \ref{th2}.

{\bf Case DBF4: $v_0\notin S_1$ and $Y_0=\set{y_0}$:}\\
If $|X|=1$ then \eqref{DB1aeq1} implies that $np_+\teps s(1+d(y_0)/(s+d(y_0)-1)$ and $np_-=1$ and if $|X|>1$ then we (i) either remain in Case DB1e while $|X|\leq \om_0^{1/2}$ or (ii) $y_0$ moves to $X$ and we are in Case DB1c or DB1d, depending on $d(y_0)$. The probability that we switch from (i) to (ii) is asymptotically equal to $\xi_j=d(y_0)/(np(j-1)+d(y_0))$, where $j=|X|$. The recurrence for $p_j$ is
\[
p_j=\l_j(1-\xi_j)p_{j+1}+\l_j\xi_j\psi_{j}+(1-\l_j)(1-\eta_j)p_{j-1}+(1-\l_j)\eta_jq_j
\]
where $\l_j=(s(j+1/(d(y_0)+s-1)))/((s(j+1/d(y_0)+s-1)+j)$ (We have $\a=1$ in the definition of $\eta_j$ in \eqref{pj}, since $v_0\notin S_1$.) We have $\psi_j=p_j(DBF2)$ if $d(y_0)\gg\e^{-2}$ and $\psi_j=p_{j,1}$ if $d(y_0)=O(\e^{-2})$.  

{\bf Case DBF5: $v_0\in S_1$ and $Y_0=\set{y_0}$:}\\
In this case we begin in Case DB1a with $d(v_0)=\a np$ where $\a=\Omega(1)$. Then we stay in Case DB1f while $|X|\leq \om_0^{1/2}$, unless $v_0$ leaves $X$. In which case we move to Case DB1e. The recurrence for $p_j$ is
\[
p_j=(1-\m_j)p_{j+1}+\m_j(1-\eta_j)p_{j-1}+\m_j\eta_jq_j
\]
where $\m_j=(s(\a+1/(d(y_0)+s-1)))/((s(j-1+\a+1/d(y_0)+s-1)+j)$ is asymptotically equal to the probability that $v_0$ leaves $X$ given that $|X|$ decreases and $\eta_j$ is as in \eqref{pj}. This establishes part (c) of Theorem \ref{th2}.

We see from the above cases that when $|X|$ is small the chance that $X$ reaches $\om_0^{1/2}$ yields (a), (b) of Theorem \ref{th2}, because if $|X|$ reaches $\om_0$ then there is a positive rightward bias and $X$ will w.h.p. eventually become $[n]$. 

\paragraph{The case $s\leq1$}
The above analysis holds for $s>1$. For $s\leq 1$ we go back to the case where $|X|\leq \om_0$. If $s<1$ then we see from \eqref{DB1x} -- \eqref{DB1c} that there are constants $C_1>0,0<C_2<1$ such that if $|X|\geq C_1$ then $p_+/p_-\leq C_2$. It follows that w.h.p. $|X|$ will return to $C_2$ before it reaches $\om_0^{1/2}$ and then there is a probability bounded away from 0 that $|X|$ will go directly to 0. 
\paragraph{The case $s=1$}
It follows from Maciejewski \cite{Mac} that the fixation probability of vertex $v$ is precisely $\p(v)=d(v)/\sum_{w\in [n]}d(w)$.  In a random graph with $np=\Omega(\log n)$ this gives $\max_v\p(v) =O(1/n)$.
\subsection{$np\gg\log n$ and $s>1$}\label{npgglogn}
When $np\gg\log n$ then $S_1=\es$ and and all but (\ref{e(S)3|S|}), (\ref{iteme(S)1|S|}), (\ref{itemBk}), \eqref{itemtheta} of Lemma \ref{lem1} hold trivially. Now \eqref{e(S)3|S|} and  (\ref{itemBk}) are used in bounding $e(X:\bar X)$ and are not therefore needed. \eqref{iteme(S)1|S|} is not used in Death-Birth. The proof  of \eqref{itemtheta} does not need $np=O(\log n)$. There is always a bias close to $s$ and the fixation probability is asymptotic to $\tfrac{s-1}{s}$.

\appendix
\section{Proof of Lemma \ref{lem1}}\label{A}
\eqref{item_max-degree}  The degree $d(v)$ of vertex $v\in [n]$ is distributed as $Bin(n-1,p)$. The Chernoff bound \eqref{ch3} implies that 
\[
\Pr(\D>5np)\leq n\Pr(Bin(n,p)\geq 5np)\leq n\bfrac{e}{5}^{5np}=o(1).
\]
\eqref{itemsizeS1} 
We first observe that the Chernoff bounds \eqref{ch1}, \eqref{ch2} imply that
\beq{eps}{
\Pr(B(n,p)\notin I_\e)=\sum_{i\notin I_\e}\binom{n}{i}p^i\brac{1-p}^{n-i}\leq e^{-\e^2np/(3+o(1))}.
}
The degree $d(v)$ of vertex $v\in [n]$ is distributed as $Bin(n-1,p)$. So, 
\[
\E(|S_1|)=n\Pr(d(1)\notin I_\e)\leq ne^{-\e^2(n-1)p/(3+o(1))} =n^{1-\e^2/(3+o(1))}.
\]
Now use the Markov inequality to obtain the upper bound. 

\eqref{itemCcapS1}
\begin{align*}
\Pr(\exists v\in S_1\cap C:\neg \eqref{itemCcapS1}))&\leq \sum_{k=3}^{\om_0}k\binom{n}{k}k!p^k\Pr(B(n-3,p)\notin I_\e-2)\\
&\leq 2(np)^{\om_0}e^{-\e^2np/(3+o(1))}\\
&=o(1).
\end{align*}
{\bf Explanation:} we sum over possible choices for a $k$-cycle $C$ of $K_n$. There are less than $\binom{n}{k}k!$ $k$-cycles in $K_n$. There are $k$ choices for a vertex of $C\cap S_1$. Given a cycle $C$ and $v\in C$ we multiply by the probability that the edges of $C$ exist in $G_{n,p}$ and that $(d_{\bar C}(v)+2)\notin I_\e$.

\eqref{itemS0S0dist} 
\mults{
\Pr(\exists v,w:\neg\eqref{itemS0S0dist})\leq \binom{n}{2}\sum_{k=1}^{\om_0-1}\binom{n-2}{k}k!p^{k+1}\brac{\sum_{i=0}^{np/10}\binom{n-k-2}{i}p^i\brac{1-p}^{n-k-2-i}}^2\leq\\ n(np)^{\om_0+1}\brac{\sum_{i=0}^{np/10}\bfrac{nep}{i(1-p)}^ie^{-np}}^2\leq n(np)^{\om_0+1} (2(10e)^{np/10}e^{-np})^2\leq n^{1+1/10+2/3-2+o(1)}=o(1).
}
{\bf Explanation:} we sum over pairs of vertices $x,y$ and paths $P$ of length $k<\om_0$ joining $x,y$. Then we multiply by the probability that these paths exist and then by the probability that $x,y$ have few neighbors outside $P$. 

\eqref{itemS0S1dist}
\mults{
\Pr(\exists x,y:\neg\eqref{itemS0S1dist})\leq n^2\sum_{\ell=1}^{\om_0}n^{\ell-1} p^\ell e^{-\e^2np/4}\sum_{i=0}^{\om_0}\binom{n-\ell-2}{i}p^i(1-p)^{n-\ell-2-i}\\
\leq 2n(np)^{2\om_0+1}e^{-(\e^2np/4+np)}=o(1).
}
{\bf Explanation:} we use a similar analysis as for property \eqref{itemS0S0dist}. The factor $e^{-\e^2np/4}$ bounds the probability that $v$ has between $(1-\e)np-\om_0$ and $(1+\e)np$ neighbors outside the chosen path vertices, see \eqref{eps}. The sum over $i$ bounds the probability that $w$ has fewer than $\om_0$ neighbors outside $v$ and the path.

\eqref{e(S)3|S|}
\[
\Pr(\exists S:\neg\eqref{e(S)3|S|})\leq \sum_{s=20}^{2n/(np)^{9/8}}\binom{n}{s}\binom{\binom{s}{2}}{10s}p^{10s}\leq \sum_{s=20}^{2n/(np)^{9/8}}e \brac{\bfrac{s}{n}^{9}\bfrac{enp}{20}^{10}}^s=o(1).
\]
{\bf Explanation:} we choose a set of size $s$ and bound the probability it has $10s$ edges by the expected number of sets of $10s$ edges that it contains. The final claim uses that fact that $np=O(\log n)$.

\eqref{iteme(S)1|S|}
\[
\Pr(\exists S:\neg\eqref{iteme(S)1|S|})\leq\sum_{s=4}^{2\om_0}\binom{n}{s}\binom{\binom{s}{2}}{s+1}p^{s+1}\leq \om_0ep\sum_{s=4}^{2\om_0}\bfrac{neps}{2}^s=o(1).
\]
We use a similar analysis as for property \eqref{e(S)3|S|}. The final claim also uses that fact that $np=O(\log n)$, in which case $(nep\om_0)^{2\om_0}\leq n^{o(1)}$. 

\eqref{iteme(S,T)}
At least one of $S,T$ has size at most $n/2$ and assume it is $S$. Suppose first that $S$ induces a connected subgraph. Suppose that $|S|\leq n/(np)^{9/8}$. We first note that
\[
n-|T|\leq \frac{n}{(np)^{9/8}}+n^{1-\e^2/4}\leq \e^2|T|.
\]
Then we have
\[
e(S:T)\leq (1+\e)|S|np=(1+\e)|S|\,|T|p+(1+\e)|S|(n-|T|)p\leq (1+\e)|S|\,|T|p(1+\e^2)\leq (1+2\e)|S|\,|T|p.
\]
On the other hand, Lemma \ref{lem1}\eqref{e(S)3|S|} implies that 
\[
e(S:T)\geq (1-\e)|S|np-20|S|\geq (1-2\e)|S|np\geq (1-2\e)|S|\,|T|p.
\]
So now assume that $n/(np)^{9/8}\leq |S|\leq n/2$. Let $\wh I_{(d)}=[n/(np)^{9/8},n/2]$. Fix $S_1$ and all edges incident with $S_1$. Then, where $m$ stands for $|S_1|$ and $n_\e=n^{1-\e^2/4}$,
\begin{align}
&\Pr\brac{\exists |S|\in {\wh I_{(d)}},e(S:T)\leq (1-\e)|S|\,|T|p} \nn\\
&\leq \sum_{s\in {\wh I_{(d)}}}\binom{n-m}{s}s^{s-2}p^{s-1}e^{-\e^2s(n-s-m)p/3}\label{ddd}\\
&\leq  \sum_{s\in {\wh I_{(d)}}}\binom{n-m}{s}s^{s-2}p^{s-1}e^{-\e^2snp/7}\nn\\
&\leq \frac{1}{s^2p}\sum_{s\in {\wh I_{(d)}}}e^{-\e^2nps/8}\nn\\
&=o(1).\nn
\end{align}
{\bf Explanation of \eqref{ddd}:} Given $s$ there are $\binom{n}{s}$ choices for $S$, $s^{s-2}$ choices for a spanning tree $T$ of $S$. The factor $p^{s-1}$ accounts for the probability that $T$ exists in $G_{n,p}$ and then the final factor $e^{-\e^2s(n-s-m)p/3}$ comes from using the Chernoff bounds to bound the probability of the event $\cA$ that $e(S:T)\leq (1-\e)|S|\,|T|p$, since $e(S:\bar S)$ is distributed as $Bin(s(n-s),p)$. These are computed conditional on the event $\cB$ that each $v\in S$ having a lower bound on its degree. Because $\cA$ is a monotone decreasing event and $\cB$ is a monotone increasing event, we can apply the FKG inequality to argue that $\Pr(\cA\mid\cB)\leq \Pr(\cA)$. (The use of the FKG inequality, or rather Harris's inequality is explained in \cite{FK}, Section 26.3.)

When it comes to estimating $\Pr\brac{\exists |S|\in {\wh I_{(d)}},e(S:T)\geq (1+\e)|S|\,|T|p}$ we apply a similar argument, but this time when we apply the FKG inequality we use the fact that each vertex has an upper bound on its degree.

We now deal with the connectivity assumption. Suppose now that $S$ has a component $C$ of size less than $10/\e^3$. Then, using Lemma \ref{lem1}\eqref{iteme(S)1|S|}, we see that w.h.p. $|N(C)|\geq d(C)-2|C|\geq (1-2\e)|C|np$ since $S\cap S_1=\es$. Clearly $|N(C)|\leq (1+2\e)|C|np$, since $C\cap S_1=\es$. So, $S$ will inherit the required property from its components. Indeed, if the components of $S$ are $C_1,C_2,\ldots,C_k$ then because there are no edges between components, 
\[
e(S:T)=\sum_{i=1}^ke(C_i:T)\geq \sum_{i=1}^k(1-2\e)|C_i||T|p=(1-\e)|S||T|p.
\]
The upper bound on $e(S:T)$ is proved in the same way.

\eqref{iteme(SSbar)}
\[
\Pr(\exists S:e(S:\bar S)\leq |S|np/2)\leq \sum_{s=\om_0/2}^{n/(np)^{9/8}}\binom{n}{s}s^{s-2}p^{s-1}e^{-s(n-s)p/3}\leq \frac{1}{s^2p}\sum_{s=\om_0/2}^{n/(np)^{9/8}}\brac{e^{1-np/4}np}^s=o(1).
\]
{\bf Explanation:} the sum bounds the expected number of spanning trees in components $S$ for which $e(S:\bar S)\leq |S|np/2$.

\eqref{item7/eps^2}
Suppose first that $|S|\leq \om_0$. Let $n_0=(5np+1)\om_0$ be an upper bound on $|S\cup N(S)|$ and let $s_0=7/\e^2$. 
Then, using \eqref{item_max-degree}, 
\begin{align*}
\Pr(\exists S:\neg\eqref{item7/eps^2})&\leq o(1)+\sum_{s=1}^{n_0}\binom{n}{s}s^{s-2}p^{s-1}\binom{s}{s_0}\binom{5s_0np}{s_0}e^{-s_0\e^2np/(3+o(1))}\nn\\
&\leq o(1)+\frac{n}{s^2p}\sum_{s=1}^{n_0}(enp)^s\bfrac{5snpe^{2-\e^2np/(3+o(1))}}{s_0^2}^{s_0}=o(1).
\end{align*}
{\bf Explanation:} here $\binom{s}{s_0}\binom{5s_0np}{s_0}$ bounds the number of choices for up to $7/\e^2$ vertices in $(S\cup N(S))\cap S_1$.

When $|S|>\om_0$ we replace $s_0$ by $s_1=\rdup{7s/\e^2\om_0}$ to obtain
\begin{align*}
\Pr(\exists S:\neg\eqref{item7/eps^2})&\leq\sum_{s=\om_0}^{n_0}\binom{n}{s}s^{s-2}p^{s-1}\binom{s}{s_1}\binom{5s_1np}{s_1}e^{-s_1\e^2np/(3+o(1))}\nn\\
&\leq \frac{n}{s^2p}\sum_{s=\om_0}^{n_0}(enp)^s\bfrac{5snpe^{2-\e^2np/(3+o(1))}}{s_1}^{s_1}=o(1).
\end{align*}
\eqref{itemBk} We use $\binom{s}{k}p^k$ to bound the probability that $v\in B_k(S)$.
\begin{align*}
\Pr(\exists S:\neg \eqref{itemBk})&\leq \sum_{k=2}^{(np)^{1/3}}\sum_{s=k}^{n/(np)^{9/8}}\binom{n}{s}s^{s-2}p^{s-1}\binom{n-s}{\a_ksnp}\brac{\binom{s}{k}p^k}^{\a_ksnp}\\
&\leq \sum_{k=2}^{(np)^{1/3}}\frac{1}{s^2p}\sum_{s=k}^{n/(np)^{9/8}}\brac{(enp)\cdot\bfrac{e^{k+1}(sp)^{k-1}}{k^k\a_k}^{\e np/k^2}}^s=o(1).
\end{align*}
\eqref{itemtheta} We can assume that $S$ induces a connected subgraph and then sum the contributions from each component. We first consider the case where $|S|\leq n/2$.
\begin{align*}
\Pr(\exists S:\neg\eqref{itemtheta} )&\leq\sum_{s=n/(np)^2}^{n/2}\binom{n}{s}s^{s-2}p^{s-1}\binom{s}{\th s}(2e^{-\e^2(n-s)p/3})^{\th s}\\
&\leq \frac{1}{p}\sum_{s=n/(np)^2}^{n/2}\brac{nep \bfrac{2e^{1-\e^2(n-s)p/3}}{\th}^{\th} }^s=o(1).
\end{align*}
When $n/2<|S|\leq n_1$ we drop the connectivity constraint and replace $\binom{n}{s}$ by $4^s$. The summand is then equal to  $\brac{4e(2e^{-\e^2(n-s)p/3}/\th)^{\th}}^s$. 

\eqref{itemSTD}
Here $\a=(np)^{1/4}$.
\begin{align*}
\Pr(\exists S:\neg\eqref{itemSTD})&\leq\sum_{s=n/(np)^{9/8}}^{n/(np)^{1/3}}\binom{n}{s}\binom{n}{\th(n-s)}\binom{\th s(n-s)}{\a\th s(n-s)p}p^{\a\th s(n-s)p}\\
&\leq \sum_{s=n/(np)^{9/8}}^{n/(np)^{1/3}}\brac{\frac{ne}{s}\cdot \bfrac{e}{\a}^{\a\th(n-s)p/2}}^s\brac{\frac{ne}{\th(n-s)}\cdot \bfrac{e}{\a}^{\a sp/2}}^{\th(n-s)}=o(1).
\end{align*}

\eqref{itemSvd}

Let $\s=\e^{-2}\log\log n$.
\beq{lastone}{
\Pr(\exists S,v:\neg\eqref{itemSvd})\leq n\sum_{s=1}^{np}\binom{n}{s}s^{s-2}p^{s-1}\binom{s}{\s}p^{\s}\leq \sum_{s=1}^{np}n^{s+1}e^s2^sp^{s+\s-1}\leq 2n^{np+1}(2e)^{np}p^{np+\s-1}.
}
{\bf Explanation:} $\binom{s}{\s}$ chooses the neighbors of $v$ in $S$. The last inequality follows from $np\gg1$.

Now suppose that $np=c\log n$. Then,
\mults{
\log(RHS\eqref{lastone})=np\log(np)+\log n+(\s-1)\log p\\
\leq (c+1)\log n\log\log n-(\s-1)(\log n-O(\log\log n))\to-\infty.
}
\end{document}